\documentclass[parskip=half]{scrartcl}
\usepackage[utf8]{inputenc}
\usepackage[T1]{fontenc}
\usepackage{lmodern}
\usepackage{amsfonts}
\usepackage{titletoc}
\usepackage[colorlinks=true,linkcolor=black,urlcolor=black,citecolor=black,bookmarks]{hyperref}
\usepackage{amsmath}
\usepackage{amssymb}
\usepackage{mathrsfs}
\usepackage{mathtools}
\usepackage{amsthm}
\usepackage{tabularx}
\usepackage{enumerate}
\usepackage{thmtools}
\usepackage{tikz}
\usepackage{float}
\usepackage{comment}
\usepackage{tikz-cd} 
\usepackage[nameinlink]{cleveref}
\usepackage[onehalfspacing]{setspace}
\usepackage[style=alphabetic, giveninits=true, sorting=nyt,maxnames=4,isbn=false,url=false,doi=false]{biblatex}
\renewbibmacro{in:}{}

\declaretheoremstyle[spaceabove=12pt plus 2pt minus 4pt ,spacebelow=12pt plus 2pt minus 4pt,bodyfont=\normalfont]{scdef}
\declaretheoremstyle[spaceabove=12pt plus 2pt minus 4pt,spacebelow=8pt plus 2pt minus 4pt,bodyfont=\itshape]{scthm} 
\declaretheoremstyle[spaceabove=\topsep,spacebelow=12pt plus 2pt minus 4pt,bodyfont=\normalfont,headfont=\itshape,notefont=\itshape,postheadspace=0.3em,qed=\qedsymbol,headformat=\NAME\NOTE,notebraces=\relax]{scpf}
\declaretheorem[numbered=no,style=scthm,name=Theorem, refname={Theorem,Theorems}, Refname={Theorem,Theorem}] {theorem*}
\declaretheorem[style=scthm,numberwithin=section,name=Theorem, refname={Theorem,Theorems}, Refname={Theorem,Theorem}] {theorem} 
\declaretheorem[style=scthm,sharenumber=theorem, name=Lemma, refname={Lemma,Lemmata}, Refname={Lemma,Lemma}] {lemma} 
\declaretheorem[style=scthm,sharenumber=theorem, name=Corollary, refname={Corollary,Corollaries}, Refname={Corollary,Corollaries}] {corollary}
\declaretheorem[style=scdef,sharenumber=theorem, name=Remark, refname={Remark,Remarks}, Refname={Remark,Remarks}] {remark}
\declaretheorem[style=scdef, sharenumber=theorem, name=Definition, refname={Definition,Definitions}, Refname={Definition,Definitions}]{definition}

\declaretheorem[style=scdef, sharenumber=theorem, name=Proposition, refname={Proposition,Propositions}, Refname={Proposition,Propositions}]{proposition}
\declaretheorem[style=scdef, sharenumber=theorem, name=Example, refname={Example,Examples}, Refname={Example,Examples}]{example}

\newcommand{\ad}{\mathrm{ad}}
\newcommand{\Ad}{\mathrm{Ad}}

\newcommand{\Aut}{\mathrm{Aut}}
\newcommand{\Isom}{\mathrm{Isom}}

\newcommand{\mf}{\mathfrak}
\newcommand{\mc}{\mathcal}
\newcommand{\R}{\mathbb{R}}

\renewcommand{\H}{\mathbb{H}}

\newcommand\extalg{%
  \newlength{\len}%
  \settoheight{\len}{V}%
  \mathbin{%
    \resizebox{0.93\len}{0.93\len}{$\wedge$}%
    \kern-0.1em%
  }}%
\newcommand{\intprod}{\mathbin{\hbox to 0.7ex{%
      \kern-0.3ex
      \vrule height0.0777ex width0.971ex depth0ex
      \kern-0.055ex
      \vrule height1.165ex width0.0777ex depth0ex\hss}}%
}%

  \DeclareMathOperator{\Lie}{Lie}

\dottedcontents{section}[2.9em]{}{2.9em}{1pc}

\AtEveryBibitem{\clearlist{language}}

\title{On Degenerate 3-$\boldsymbol{(\alpha,\delta)}$-Sasakian Manifolds}
\author{Oliver Goertsches, Leon Roschig, Leander Stecker}
\date{}

\begin{document}

\maketitle

\begin{abstract}
\begin{center} \textbf{Abstract} \end{center}
We propose a new method to construct degenerate 3-$(\alpha, \delta)$-Sasakian mani- folds as fiber products of Boothby-Wang bundles over hyperkähler manifolds. Subsequently, we study homogeneous degenerate 3-$(\alpha, \delta)$-Sasakian manifolds and prove that no non-trivial compact examples exist as well as that there is exactly one family of nilpotent Lie groups with this geometry, the quaternionic Heisenberg groups.
\end{abstract}

\tableofcontents

\section{Introduction}

Almost (3-)contact metric structures and special cases thereof are among the most well-studied types of geometries in odd dimensions. Arguably the most famous examples are (3-)Sasakian manifolds. \textsc{Agricola} and \textsc{Dileo} have proposed \emph{3-$(\alpha,\delta)$-Sasakian structures} as a generalization of the latter which retain favorable properties like hypernormality or the existence of a connection with skew torsion adapted to the geometry \cite{AD}. These classes of manifolds depend on two real parameters $ \alpha, \delta \in \R $, $ \alpha \neq 0 $, and can be grouped into \emph{degenerate} and \emph{nondegenerate} type according to whether $ \delta = 0 $ or $ \delta \neq 0 $.

In this article, we focus on the degenerate case where the three contact metric structures are more independent of each other in the sense that their Reeb vector fields commute. As will be explained in \Cref{constr}, this allows us to construct examples of such manifolds via fiber products of Boothby-Wang bundles over hyperkähler manifolds. This can be viewed as a counterpart to obtaining 3-Sasakian manifolds from the classical Konishi bundle over positive quaternionic Kähler spaces.

\textsc{Agricola}, \textsc{Dileo} and \textsc{Stecker} have shown that all 3-$(\alpha, \delta)$-Sasakian manifolds locally submerge onto a quaternionic Kähler space whose curvature depends on the sign of $ \alpha\delta $ \cite[Theorem 2.2.1]{ADS}. This result motivated them to construct plentiful examples of homogeneous non-degenerate 3-$(\alpha, \delta)$-Sasakian manifolds over quaternionic Kähler spaces of non-vanishing curvature \cite[Section 3]{ADS}. To some extent, this property also explains why it could be expected that there are rather few homogeneous degenerate 3-$(\alpha,\delta)$-Sasakian manifolds. This is because, ideally, such spaces would submerge onto homogeneous hyperkähler manifolds which are known to be completely flat, i.e.~products of tori and Euclidean spaces \cite{AK}. Compared to the non-degenerate case, however, it is less clear if the corresponding hyperkähler space will actually be a manifold globally. We identify two such scenarios where the base space is indeed a manifold. However, we show that there are restrictions beyond those imposed by \cite{AK} on the base.

Unlike in the non-degenerate case, the Reeb vector fields are now elements of the automorphism algebra of a homogeneous degenerate 3-$(\alpha,\delta)$-Sasakian manifold because they commute with one another. In \Cref{hom}, we show that they comprise the center of the automorphism algebra and conclude that they generate a \emph{closed} subgroup of the automorphism group. Combining this with classical results due to \textsc{Alekseevskii} \cite{Alek} as well as \textsc{Conner} and \textsc{Raymond} \cite{CR}, we prove that the only connected \emph{compact} homogeneous degenerate 3-$(\alpha, \delta) $-Sasakian manifold is the trivial example of the 3-torus $ T^3 $ (\Cref{comp}). Finally, we study simply connected nilpotent Lie groups with a left-invariant degenerate 3-$(\alpha, \delta)$-Sasakian structure in \Cref{Lie}. Again using traditional results by \textsc{Wilson} \cite{Wil} as well as \textsc{Milnor} \cite{Mil}, we show there is exactly one family of such spaces, the \emph{quaternionic Heisenberg groups}. This parallels a result about nilpotent Sasakian Lie groups by \textsc{Andrada}, \textsc{Fino} and \textsc{Vezzoni} \cite[Theorem 3.9]{AFV}.

To sum up, our findings paint the somewhat curious picture that there are no non-trivial compact homogeneous examples in the degenerate 3-$(\alpha,\delta)$-Sasakian category, despite there being both interesting non-compact homogeneous (quaternionic Heisenberg group) and compact inhomogeneous cases ($3$-Boothby-Wang bundles over compact inhomogeneous hyperkähler manifolds, see \Cref{constr}). Qualitatively, this geometry is in stark contrast to the classical 3-Sasakian one where homogeneous examples are automatically compact and there exist as many as compact simple Lie groups (\cite[Theorem C]{BGM}, see also \cite{GRS}).

\textbf{Acknowledgements:} The second named author is supported by the German Academic Scholarship Foundation. The authors thank Ilka Agricola for fruitful discussions about the subject.

\section{Preliminaries}
Let $(M^{2n+1},g)$ be a Riemannian manifold of odd dimension. An \emph{almost contact metric structure} $(\xi,\eta,\varphi)$ is defined by a unit length vector field $\xi\in\mathfrak{X}(M)$, its metric dual $\eta\in \Omega^1(M)$ and an almost hermitian structure $\varphi\in \mathrm{End}(TM)$ on $\ker \eta$ satisfying
\begin{gather*}
\varphi\, \xi=0,\qquad \varphi^2=-\mathrm{id}+\xi\otimes\eta,\qquad g(\varphi X,\varphi Y)=g(X,Y)-\eta(X)\eta(Y) \, , \quad X ,Y \in TM \,.
\end{gather*}
For an almost contact metric manifold $(M,g,\xi,\eta,\varphi)$, we let $\Phi\in\Omega^2(M)$ denote the \emph{fundamental 2-form} given by $\Phi(X,Y)=g(X,\varphi Y)$.

Now suppose that $(M^{4n+3},g)$ is equipped with three almost contact metric structures $(\xi_i,\eta_i,\varphi_i)_{i=1,2,3}$ which are compatible in the sense that
\begin{gather*}
\varphi_i \, \xi_j=\xi_k,\qquad \eta_i\circ\varphi_j=\eta_k, \qquad \varphi_i \circ \varphi_j=\varphi_k+\xi_i\otimes\eta_j
\end{gather*}
for any even permutation $(i\,j\,k)$ of $(1\,2\,3)$. Then, we call $(M,g,\xi_i,\eta_i,\varphi_i)_{i=1,2,3}$ an \emph{almost $3$-contact metric manifold}. By construction, the Reeb vector fields $\xi_i$ are orthogonal and form a $3$-dimensional distribution $\mathcal{V}=\langle\xi_1,\xi_2,\xi_3\rangle$. We call this distribution \emph{vertical} and its orthogonal complement $\mathcal{H}=\bigcap_i \ker\eta_i$ \emph{horizontal}.

\begin{samepage}
\begin{definition}(\cite[Definition 2.2.1]{AD}).
A \emph{$3$-$(\alpha,\delta)$-Sasakian manifold}, $\alpha,\delta\in \R$, $\alpha\neq 0$, is an almost $3$-contact metric manifold satisfying
\begin{align}\label{3addef}
d\eta_i=2\alpha\Phi_i+2(\alpha-\delta)\eta_j\wedge\eta_k
\end{align}
for any even permutation $(i\,j\,k)$ of $(1\,2\,3)$.
\end{definition}
\end{samepage}

In case $\alpha=\delta=1$, the differential condition is just $d\eta_i=2\Phi_i$, which corresponds to standard $3$-Sasakian manifolds. We will recall only those results about $3$-$(\alpha,\delta)$-Sasakian manifolds which we need here. For a more complete survey we refer to the introductory article \cite{AD} as well as the third author's thesis \cite{LeanderDiss}.

Among $3$-$(\alpha,\delta)$-Sasakian manifolds we differentiate two subclasses due to the behavior of their Reeb vector fields.

\begin{proposition}(\cite[Corollary 2.3.1]{AD}.)
The Reeb vector fields are Killing and for any even permutation $(i\,j\,k)$ of $(1\,2\,3)$ they satisfy the commutator relations
\begin{align*}
[\xi_i,\xi_j]=2\delta\xi_k.
\end{align*}
In particular, the vertical distribution is integrable.
\end{proposition}
If $\delta\neq 0$, then the Reeb vector fields span a Lie algebra isomorphic to $\mathfrak{su}(2)$. For $\delta=0$ on the other hand, the vertical Lie algebra is Abelian.
\begin{definition}
A $3$-$(\alpha,\delta)$-Sasakian manifold is called \emph{degenerate} if $\delta=0$ and \emph{non-degenerate} otherwise.
\end{definition}
We write $\Phi_i^\mathcal{H}:=\Phi_i \circ (\pi_\mathcal{H} \otimes \pi_\mathcal{H}) $, where $ \pi_\mc{H}: TM \to \mathcal{H} $ denotes the orthogonal projection. Then for degenerate $3$-$(\alpha,\delta)$-Sasakian manifolds \eqref{3addef} assumes the simpler form
\begin{align}
d\eta_i=2\alpha\Phi_i^\mathcal{H} \, . \label{degencond}
\end{align}
In particular, we have $\ker d\eta_i=\mathcal{V}$.
%
%

\section{A Constructive Result} \label{constr}
Whereas plentiful examples of non-degenerate $3$-$(\alpha,\delta)$-Sasakian structures are known on Konishi bundles over quaternionic Kähler orbifolds of non-vanishing scalar curvature \cite[Section 3]{ADS}, the only known examples of degenerate $3$-$(\alpha,\delta)$-Sasakian manifolds are the quaternionic Heisenberg groups and derived examples (\cite[Subsection 2.2]{AFS}, \cite[Example 2.3.2]{AD}, see also \Cref{heisenberg}). The following construction starting with a hyperkähler manifold is an excerpt form the third author's thesis \cite{LeanderDiss}.

Let $(N,g,I_1,I_2,I_3)$ be a hyperkähler manifold. Suppose that the fundamental two-forms are integer classes $[\omega_i]\in H^2(N,\mathbb{Z})$. Then we obtain a Boothby-Wang bundle $S^1\to P_i\xrightarrow{\ \pi_i\ } N$ for each Kähler structure \cite[Theorem 3]{BW}. In particular, these bundles have Chern class $[\omega_i]$ and the total spaces $P_i$ are equipped with Sasaki structures $(g_i,\tilde{\varphi}_i,\eta_i,\xi_i)$ such that $\xi_i$ generates the fiber.

Now consider the product bundle 
\[
T^3\longrightarrow P_1\times P_2\times P_3\xrightarrow{\ \Pi=(\pi_1,\pi_2,\pi_3)\ } N^3.
\]
Let $M:=\Pi^{-1}(\Delta N)$ denote the restriction of the product bundle $P_1\times P_2\times P_3$ to the diagonal \mbox{$\Delta(N)=\{(x,x,x)\in N^3\}$} and consider $M$ as a fiber bundle $M\xrightarrow{\ \pi\ }N$, where $\pi$ is the composition 
\[
P_1\times P_2\times P_3\supset M\xrightarrow{\ \Pi\ } \Delta(N)\xrightarrow{\ \cong\ } N,\quad (p_1,p_2,p_3)\mapsto (x,x,x)\mapsto x.
\]
If we denote the fiber of $P_i$ over $x\in N$ by $(P_i)_x\coloneqq \pi_i^{-1}(\{x\})\cong S^1$, then the bundle $M\xrightarrow{\pi}N$ can be seen as the fiber bundle with fiber $(P_1)_x\times(P_2)_x\times (P_3)_x$ over any given point $x\in N$. For $p=(p_1,p_2,p_3)\in M$ we have 
\[
T_pM\cong T_xN\oplus \langle\xi_1,\xi_2,\xi_3\rangle
\]
as $\xi_i$ generates the fiber of $P_i$.
The construction as a fiber product bundle ensures that the Reeb vector fields $\xi_i$ are linearly independent and $[\xi_i,\xi_j]=0$, as it ought to be for degenerate $3$-$(\alpha,\delta)$-Sasakian manifolds.

Extend $\eta_i$ trivially, i.e.~$\ker\eta_i=TN\oplus \langle \xi_j,\xi_k\rangle$, with $j,k\neq i$, and $\eta_i(\xi_i)=1$. Set the metric $g\coloneqq \pi^*g_N+\eta_1^2+\eta_2^2+\eta_3^2$, then $\pi\colon (M,g)\to (N,g_N)$ becomes a Riemannian submersion.
\begin{theorem}\label{3BW}
Let $(N,g_N,I_1,I_2,I_3)$ be a hyperkähler manifold with integer fundamental forms. Construct $(M,g,\eta_i,\xi_i)$ as above and set $\varphi_i=\tilde{\varphi}_i+\eta_j\otimes\xi_k-\eta_k\otimes \xi_j$ for any cyclic permutation $(i\,j\,k)$ of $(1\,2\,3)$.\\
Then $(M,g,\xi_i,\eta_i,\varphi_i)_{i=1,2,3} $ is a degenerate $3$-$(\alpha,\delta)$-Sasakian manifold with $\alpha=1$, $\delta=0$.
\end{theorem}
\begin{remark}
The $T^3$-bundle $M$ also appears in \textsc{Fowdar}'s paper \cite{Fowdar} in his construction of quaternionic Kähler metrics on $\R\times M$ albeit without the geometric structure, see also \cite[Lemma 2.1]{Cor}.

In \cite{For} \textsc{Foreman} obtains a so-called \emph{complex contact manifold} by constructing a $T^2$-bundle over $N$ in similar fashion. His construction imposes a less restrictive assumption than hyperkähler on $N$. However, in this special case the complex $T^2$-bundle he considers can be obtained from our construction as the quotient by one of the Reeb vector fields. This should be considered as the analogue of the \emph{twistor space} over a positive quaternionic Kähler manifold and the $3$-Sasakian space above (cmp.~\cite{For2}).
\end{remark}
\begin{proof}
Observe that $(g,\eta_i,\xi_i,\varphi_i)$ extends the almost contact metric structure on $T_{p_i}P_i\cong T_xN\oplus \langle\xi_i\rangle$ to $T_pM\cong T_xN\oplus \langle \xi_i,\xi_j,\xi_k\rangle$. The three almost contact metric structures are compatible on the vertical subspace by definition of $\varphi_i$. On the horizontal distribution $\varphi_i$ projects to $I_i$, thus $\varphi_i \circ \varphi_j|_{\mathcal{H}}=\varphi_k|_\mathcal{H}$.

The last identity to check is the differential condition \eqref{degencond}. As $(\xi_i,\eta_i,\tilde{\varphi}_i)$ defines a Sasakian structure on $P_i$ we have $g_i(X,\tilde{\varphi}_iY)=d\tilde{\eta}_i(X,Y)$  for vertical vectors $X,Y\in T_{p_i}P_i$. This remains true for $X,Y\in\mathcal{H}\subset T_pM$ as $\varphi_i$ is just $\tilde{\varphi}_i$ on $\mathcal{H}$. If either vector of $X,Y$ \mbox{lies in $ \mathcal{V}$,} the left hand side of \eqref{degencond} has to vanish. We have
\[
\xi_j\intprod d\eta_i=L_{\xi_j}\eta_i-d(\xi_j\intprod \eta_i)=0
\]
as $\eta_i$ is defined on the factor $P_i$ of the product $P$. In particular, it is invariant under $\xi_j$ for $j\neq i$. Finally, $\xi_i\intprod d\eta_i=0$ from the Sasakian condition on $P_i$.
\end{proof}

We call this bundle the \emph{$3$-Boothby-Wang bundle} over $N$.
As shown in \cite[Theorem 2.2.1]{ADS}, locally any degenerate $3$-$(\alpha,\delta)$-Sasakian manifold maps onto a hyperkähler space via the so-called \emph{canonical submersion}. The above construction describes the converse starting from a hyperkähler manifold. By construction, the Reeb vector fields span the fiber of the bundle $T^3\to M\xrightarrow{\pi} N$. Therefore the canonical submersion agrees globally with $\pi$ mapping onto a hyperkähler manifold, namely the initial manifold $N$. Observe that starting with a compact hyperkähler manifold in this way one obtains a compact degenerate $3$-$(\alpha,\delta)$-Sasakian manifold.

\begin{remark}
In \cite{Cor} \textsc{Cortés} shows that non-flat compact hyperkähler manifolds of arbitrary dimension with integral Kähler forms exist. This implies that there are compact degenerate $3$-$(\alpha,\delta)$-Sasaki manifolds that are not quotients of a quaternionic Heisenberg group.
\end{remark}

\begin{example} \label{3dim}
As a toy example, we could consider $ N $ to be a single point, meaning that $ M $ is just $ T^3 $. It turns out that this is the only connected compact degene- rate 3-$(\alpha,\delta)$-Sasakian manifold in dimension three, which can be seen as follows: If $ (M^3,g, \xi_i,\eta_i,\varphi_i)_{i=1,2,3} $ is a compact connected degenerate 3-$(\alpha,\delta)$-Sasakian 3-manifold, then the Reeb vector fields generate an $ \R^3 $-action on $ M $. The isotropy groups of this action are discrete (since the Reeb vector field vanish nowhere), so any orbit is three-dimensional and thus, by connectedness, equal to all of $ M $. Hence, $ M = G/H $ is a connected compact Abelian Lie group ($ H $ is a normal subgroup, since $ G $ is Abelian), i.e.~a torus. The structure tensors $ (g,\xi_i,\eta_i,\varphi_i)_{i=1,2,3} $ are completely determined by the definition of almost contact 3-structures.
\end{example}

\section{Homogeneous Spaces} \label{hom}

From now on, let $ (M^{4n+3}, g, \xi_i,\eta_i,\varphi_i)_{i=1,2,3} $ be a connected degenerate 3-$(\alpha, \delta)$-Sasakian manifold. A \emph{3-$(\alpha, \delta)$-Sasakian automorphism} of $ M $ is an isometry $ \phi: M \to M $ which satisfies any of the equivalent conditions $ \phi_\ast \xi_i = \xi_i $, $ \phi^*\eta_i =  \eta_i $ or $ \phi \circ \varphi_i = \varphi_i \circ \phi $, $ i =1,2,3 $. Let us now assume that $ M $ is \emph{homogeneous}, i.e.~the identity component $ G:= \Aut_0(M) $ of its group $ \Aut(M) $ of 3-$(\alpha,\delta)$-Sasakian automorphisms acts transitively. The Lie algebra $ \mf{g} $ is (anti-)isomorphic to the space $ \mf{aut}(M) $ of all complete Killing vector fields on $ M $ which commute with the Reeb vector fields $ \xi_1,\xi_2,\xi_3 $ via the map $ \mf{g} \to \mf{aut}(M), \, X \mapsto \overline{X} $, where $ \overline{X} $ denotes the fundamental vector field of the left $ G $-action.

Since the automorphism group acts transitively, the Reeb vector fields are complete. Unlike in the 3-Sasakian and more generally the non-degenerate 3-$(\alpha,\delta)$-Sasakian case, the Reeb vector fields commute with each other, so we have $ \xi_1,\xi_2,\xi_3 \in \mf{aut}(M) $. In fact, the corresponding elements $ X_1,X_2,X_3 \in \mf{g} $ (i.e.~$ \overline{X_i}=\xi_i $) clearly even lie in the center $ Z(\mf{g}) $ of $ \mf{g} $. We will show that $ \mf{g} $ contains no other central elements using techniques very similar to those in \cite[Section 6]{GRS}.

\begin{lemma} \label{lieder}
\[ d\eta_i(\overline{X},\overline{Y}) = \eta_i ([\overline{X}, \overline{Y}]) \qquad i =1,2,3 ,\; X,Y \in \mf{g} \, . \]
\end{lemma}
\begin{proof}
Evaluating the exterior derivative gives
\[ d\eta_i(\overline{X},\overline{Y}) = \overline{X}\big(\eta_i(\overline{Y})\big) - \overline{Y}\big(\eta_i(\overline{X})\big) - \eta_i ([\overline{X},\overline{Y}]) \, . \]
and the Leibniz rule for the Lie derivative implies
\[ \overline{X}\big(\eta_i(\overline{Y})\big) = \mathcal{L}_{\overline{X}} \big(\eta_i(\overline{Y})\big) = \big(\mathcal{L}_{\overline{X}}\eta_i \big) (\overline{Y}) + \eta_i\big(\mathcal{L}_{\overline{X}}\overline{Y}\big) = \eta_i ([\overline{X},\overline{Y}])  \, , \]
where $ \mathcal{L}_{\overline{X}}\eta_i = 0 $ because $ G $ acts by 3-$(\alpha,\delta)$-Sasakian automorphisms.
\end{proof}

\begin{lemma} \label{const}
A vector field of the form $ \sum_{i=1}^3 f_i \, \xi_i $, where $ f_1, f_2,f_3 \in C^\infty(M) $, lies in $ \mf{aut}(M) $ if and only if $ f_1, f_2, f_3 $ are constant functions.
\end{lemma}
\begin{proof}
If $ V = \sum_{i=1}^3 f_i \,\xi_i \in \mf{aut}(M) $, then for $ j =1,2,3 $ and $ W \in \mf{X}(M) $, we have:
\[ 0 = \big(\mathcal{L}_V \eta_j\big)(W) = \sum_{i=1}^3 \Big(f_i \underbrace{(\mathcal{L}_{\xi_i} \eta_j)}_{=0}(W) - W(f_i) \underbrace{\eta_j(\xi_i)}_{=\delta_{ij}}\Big) = -W(f_j) \, . \]
Since $ M $ is connected, $ f_1,f_2,f_3 $ have to be constant.
\end{proof}

\begin{proposition} \label{centerAut}
$ Z(\mf{g}) = \mf{k} := \langle X_1,X_2,X_3 \rangle $.
\end{proposition}
\begin{proof}
Let $ X \in Z(\mf{g}) $ and choose some $ p \in M $ and $ Y \in \mf{g} $ such that $ \overline{Y}_p = \varphi_1 \overline{X}_p $ (which is possible, since $ G $ acts transitively). Then, by virtue of \autoref{lieder}:
\begin{align*}
0 = \eta_1(\overline{[X,Y]}_p) = -d\eta_1(\overline{X}_p,\varphi_1\overline{X}_p) = 2\alpha \left\|(\overline{X}_p)_\mathcal{H}\right\|^2 \, ,
\end{align*}
where $ v_\mathcal{H} $ denotes the horizontal component of a tangent vector $ v \in TM $. Since $ \alpha \neq 0 $ and $ p \in M $ was arbitrary, the vector field $ \overline{X} $ has to be completely vertical and therefore of the form described in \autoref{const}. Consequently, $ \overline{X} = \sum_{i=1}^3 a_i\xi_i $ for some $ a_1,a_2,a_3 \in \R $ and $ X = \sum_{i=1}^3 a_iX_i \in \mf{k} $.
\end{proof}

\begin{corollary} \label{closed}
The connected subgroup $ K \subset G $ with Lie algebra $ \mf{k}$ is closed.
\end{corollary}
\begin{proof}
Since $ G $ is connected, the Lie algebra of $ Z(G) $ is given by $ Z(\mf{g}) $. Thus, \autoref{centerAut} implies that $ K = Z_0(G) $, the identity component of the center. Since $ Z(G) $ is closed in $ G $ and $ Z_0(G) $ is closed in $ Z(G) $, it follows that $ K $ is closed in $ G $.
\end{proof}

\section{Compact Homogeneous Spaces} \label{comp}

Let us now additionally assume that $ M $ is compact. Then its isometry group $ \Isom(M) $ as well as the closed subgroups $ \Aut(M) $ and $ \Aut_0(M) = G $ are also compact, as is $ K $ by \autoref{closed}. Thus, the space $M/K$ of $K$-orbits is at least an orbifold. However, since $ M $ is $ G $-homogeneous, the $ K $-action on $ M $ only has one isotropy type. Hence, the orbit space $ M/K $ is a compact smooth manifold, in fact:

\begin{proposition}\label{Torus}
$ M/K $ is a torus.
\end{proposition}
\begin{proof} \textsc{Agricola}, \textsc{Dileo} and \textsc{Stecker} have shown that the space of leaves of the cha- racteristic foliation of a degenerate 3-$(\alpha,\delta)$-Sasakian manifold (i.e.~the foliation generated by the Reeb vector fields) locally admits a hyperkähler structure \cite[Theorem 2.2.1]{ADS}. Concretely, the hyperkähler forms are induced by $d\eta_i$ which are basic by \eqref{degencond}. Since $ X_1,X_2,X_3 $ are the infinitesimal generators of both $ K $ and the characteristic foliation, the leaf space is given by the orbit space $ M/K $. By construction, the hyperkähler structure on $ M/K $ is also $ G $-homogeneous. A result by \textsc{Alekseevskii} then concludes that $ M/K $ has to be a torus \cite[Theorem 1 b)]{Alek}.
\end{proof}

\begin{theorem}
The only connected compact homogeneous degenerate 3-$(\alpha,\delta) $-Sasakian manifold is $ T^3 $.
\end{theorem}
\begin{proof}
Let $ L $ be the kernel of the $ G $-action on the hyperkähler quotient space $ M/K $. Then, $ G/L $ is a connected compact Lie group which acts effectively on a closed aspherical manifold (i.e.~a manifold whose universal cover is contractible). It follows from the work of \textsc{Conner} and \textsc{Raymond} that $ G/L $ is also a torus \cite[Theorem 5.6]{CR}. \\
Since $ L $ is a normal subgroup of $ G $, its Lie algebra $ \mf{l} $ is an ideal in $ \mf{g} $. Because $ G $ is compact, we may choose an $ \Ad(G) $-invariant (and thus also $ \ad(\mf{g}) $-invariant) inner product on $ \mf{g} $ and consider the complementary ideal $ \mf{m} := \mf{l}^\perp $. Again by compactness of $ G $, we have $ \mf{g} = [\mf{g},\mf{g}] \oplus \mf{k} $ and $ [\mf{g},\mf{g}] = \mf{k}^\perp $ (independently of the choice of the inner product). Since $ K \subset L $, it follows that $ \mf{m} $ is an ideal inside the semisimple algebra $ [\mf{g},\mf{g}] $ and is therefore itself semisimple. But on the other hand, $ \mf{m} $ is isomorphic to the Lie algebra $ \mf{g}/\mf{l} $ of $ G/L $, which is Abelian. Consequently, $ \mf{m} = \{0\} $. \\
It follows that $ G/L $ has to be the trivial torus, meaning that $ G $ acts trivially on $ M/K $. Since this action is also transitive, $ M/K $ is a singleton and $ M $ is 3-dimensional. By virtue of \autoref{3dim}, $ M $ is isomorphic to $ T^3 $.
\end{proof}

\section{Nilpotent Lie Groups} \label{Lie}

In this section, we want to investigate simply connected nilpotent Lie groups with a left-invariant degenerate 3-$(\alpha,\delta)$-Sasakian structure. Let us first recall the specific example of the \emph{quaternionic Heisenberg group}, which was first described as a naturally reductive space in \cite{AFS} and later identified as a homogeneous degenerate 3-$(\alpha, \delta) $-Sasakian manifold in \cite{AD}.

\begin{example} \label{heisenberg}
The \emph{quaternionic Heisenberg group} of dimension $ 4n+3 $ is the simply connected Lie group $ H $ determined by the following Lie algebra $ \mf{h} $: As a set, $ \mf{h} = Z \oplus \H^n $, where $ Z $ is the span of the imaginary quaternions $ i,j,k $, which we will denote by $ \xi_1,\xi_2,\xi_3 $. We define an inner product $ g $ on $ \mf{h} $ by requiring that $ \xi_1,\xi_2,\xi_3 $ are orthonormal, $ g|_{\H^n \times \H^n} $ is the standard Euclidean inner product on $ \H^n $ and $ g|_{Z \times H^n} = 0 $. Let $ \eta_i $ be the $ g $-dual 1-form of $ \xi_i $ and $ \varphi_i $ be the endomorphism of $ \mf{h} $ such that $ \varphi_i \, \xi_j = \xi_k $ and $ \varphi_i|_{\H^n} $ is quaternionic multiplication from the left by $ \xi_i $. Finally, the Lie bracket on $ \mf{h} $ is determined by the conditions $ Z(\mf{h}) = Z $, $ [\H^n,\H^n] \subset Z $ (implying that $ \mf{h} $ is 2-step nilpotent) and
\[ g([X,Y],\xi_i) = g(\varphi_i X, Y) \quad \forall \; X,Y \in \H^n \, . \]
As shown in \cite[Subsection 2.2]{AFS} and \cite[Example 2.3.2]{AD}, the corresponding left-invariant tensor fields $ (g, \xi_i,\eta_i,\varphi_i)_{i=1,2,3} $ constitute a homogeneous degenerate 3-$(\alpha,\delta)$-Sasakian structure on $ H $ with $ \alpha = 1/2 $. As always, one can modify this structure (and in particular the value of $ \alpha$) by applying an \emph{$ \mc{H} $-homothetic deformation} \cite[Definition 2.3.1]{AD}. In fact, the presentation of the quaternionic Heisenberg group in \cite{AFS} and \cite{AD} already incorporates an $ \mc{H} $-homothetic deformation with parameters $ a = 1$, $ b = \lambda^2 -1 $ and $ c = \lambda $, where $ \lambda > 0 $.
\end{example}

From now on, let $ N^{4n+3} $ denote a simply connected nilpotent Lie group with a left-invariant degenerate 3-$(\alpha,\delta)$-Sasakian structure $ (g,\xi_i,\eta_i,\varphi_i)_{i=1,2,3} $. The Lie algebra \mbox{$\mf{n} = \mc{H} \oplus \mc{V} $} decomposes both as vector spaces and as left-invariant distributions inside the tangent bundle.

\begin{lemma} \label{center}
$ Z(\mf{n}) \subset \mc{V} $.
\end{lemma} \vspace{-.5cm}
\begin{proof}
The derivative of the left-invariant 1-form $\eta_1$ is given by $ d\eta_1(X,Y) = -\eta_1([X,Y]) $ for all $ X,Y \in \mf{n} $. Thus if $ X \in Z(\mf{n}) $, then Condition \eqref{degencond} implies that for $ Y := \varphi_1X $: \vspace{-.25cm}
\[ \vspace{-.25cm} 0 = \eta_1([X,Y]) = -d\eta_1(X,Y) = 2\alpha \|X_\mc{H}\|^2 \, . \]
\end{proof}

\begin{proposition}
$ Z(\mf{n}) = \mc{V} $ and $ \mf{n} $ is 2-step nilpotent.
\end{proposition} \vspace{-.5cm}
\begin{proof}
Let $ V $ denote the connected subgroup of $ N $ with Lie algebra $ \mc{V} $. Since $ N $ is simply connected and nilpotent, the (Lie group) exponential map of $ N $ is a global diffeomorphism \cite[Theorem 1.127]{Knap}. Hence, $ V $ is a closed subgroup and $ N/V $ is (globally) a smooth manifold. Now, $ N/V $ is also the space of leaves of the characteristic foliation of $ N $, which admits a hyperkähler structure (\cite[Corollary 2.2.1]{ADS}) that is homogeneous as argued in the proof of \autoref{Torus}. \\
Let $ W $ denote the kernel of the $ N $-action on $ N /V $. Then $ W $ is known as the \emph{normal core of $ V $} and it is the largest subgroup of $ V $ which is normal in $ N $. The connected nilpotent quotient group $ N/W $ acts effectively and transitively by isometries on $ N/V $, so $ N/V $ is isometric to $ N/W $ by virtue of a classical result due to \textsc{Wilson} \mbox{\cite[Theorem 2]{Wil}.} Consequently, $ \dim V = \dim W $, so $ \mc{V} = \Lie(W) $ is an ideal in $ \mf{n} $. But $ \eta_j([X,\xi_i]) = d\eta_j(\xi_i,X)=0 $ for all $ X \in \mf{n}$, so $ \mc{V} = Z(\mf{n}) $. Moreover, $ N/V $ is nilpotent and Ricci-flat, so a result by \textsc{Milnor} implies that its Lie algebra $ \mf{n}/\mc{V} = \mf{n}/Z(\mf{n}) $ is Abelian \cite[Theorem 2.4]{Mil}. Consequently, $ [\mf{n},\mf{n}] \subset Z(\mf{n}) $, so $ \mf{n} $ is 2-step nilpotent.
\end{proof}

\begin{theorem}
The only simply connected nilpotent Lie groups with a left-invariant degenerate 3-$(\alpha,\delta)$-Sasakian structure are the quaternionic Heisenberg groups, endowed with the structure described in \autoref{heisenberg}.
\end{theorem}
\begin{proof}
Since $ N $ is simply connected, it suffices to construct a Lie algebra isomorphism $ \psi: \mf{n} \to \mf{h} $ which preserves the structure tensors. First, let $ \psi $ map the Reeb vector fields of $ \mf{n} $ to those of $ \mf{h} $. Both $ \mc{H} \subset \mf{n} $ and $ \H^n \subset \mf{h} $ are left quaternionic vector spaces with compatible inner products, so they admit orthonormal bases of the form $ (e_1,\varphi_1e_1,\varphi_2e_1,\varphi_3e_1, \ldots, e_n, \varphi_1e_n, \varphi_2e_n,\varphi_3 e_n) $. Let $ \psi $ map such a basis of $ \mc{H} $ to a corresponding basis of $ \H^n $.

From these definitions, it is clear that $ \psi $ is a \emph{linear} isomorphism which respects the structure tensors, so it only remains to check that $ \psi $ preserves the Lie bracket. To this end, first note that by the previous proposition: $ \psi (Z(\mf{n})) = \psi (\mc{V}) = Z = Z(\mf{h}) $. Likewise, $ [\psi(\mc{H}),\psi(\mc{H})] = [\H^n,\H^n] \subset Z $ and $ \psi[\mc{H},\mc{H}] \subset \psi(\mc{V}) = Z $. Finally, for all $ X, Y \in \mc{H} $, we have (potentially after a suitable $ \mc{H} $-homothetic deformation to match the values of $ \alpha $):
\begin{align*}
\eta_i^\mf{h}([\psi X,\psi Y]) &= d \eta_i^\mf{h}(\psi Y,\psi X) = 2 \alpha g^\mf{h}(\varphi_i^\mf{h}\psi X, \psi Y) = 2\alpha g^\mf{h}(\psi \varphi_i^\mf{n} X, \psi Y) \\
&= 2\alpha g^\mf{n}(\varphi_i^\mf{n} X, Y) = d \eta_i^\mf{n}(Y,X) = \eta_i^\mf{n}([X,Y]) = \eta_i^\mf{h}(\psi[X,Y]) \, .
\end{align*}
\end{proof}

\begin{remark}
This theorem is very analogous to a result by \textsc{Andrada}, \textsc{Fino} and \textsc{Vezzoni}: They showed that the only simply connected nilpotent Lie groups with a left-invariant \emph{Sasakian} structure are the odd-dimensional Heisenberg groups \cite[Theorem 3.9]{AFV}. Finally, however, we would like to warn the reader about the somewhat counterintuitive fact that, at least in the Sasakian realm, there exist (non-nilpotent) Lie groups with a left-invariant structure where the Reeb vector field does \emph{not} lie in the center of the Lie algebra.
\end{remark}

\printbibliography

\textsc{Oliver Goertsches and Leon Roschig, Philipps-Universität Marburg, Fachbereich Mathematik und Informatik, Hans-Meerwein-Straße, 35043 Marburg} \\
\textit{E-mail addresses}: \texttt{goertsch@mathematik.uni-marburg.de}, \\
\texttt{roschig@mathematik.uni-marburg.de}

\textsc{Leander Stecker, Universität Hamburg, Fachbereich Mathematik, Bundes\-straße, 20146 Hamburg} \\
\textit{E-mail address}: \texttt{leander.stecker@uni-hamburg.de}

\end{document}